\documentclass[reqno]{amsart}
\usepackage{graphicx}
\usepackage{indentfirst,csquotes}

\usepackage{amssymb,amsthm,amsmath}
\usepackage{xcolor,paralist,hyperref,titlesec,fancyhdr,etoolbox}
\newtheorem{theorem}{Theorem}[]

\usepackage{multirow}
\newtheorem{assumption}{Assumption}

\usepackage{amsthm}
\theoremstyle{remark}
\newtheorem{remark}{Remark}

\titleformat{\section}
  {\normalfont\Large\bfseries\centering}
  {\thesection}{1em}{}

\titlespacing*{\subsection}
{0pt}      % left indent (keep 0pt => left aligned)
{1.0ex}    % space before
{0.8ex}    % space after (increase this to separate from paragraph)

\titlespacing*{\subsubsection}
{0pt}
{0.8ex} 
{0.6ex}
\titleformat{\subsection}
  {\normalfont\large\bfseries}  % style (left aligned by default)
  {\thesubsection}{1em}{}

\hypersetup{ colorlinks=true, linkcolor=black, filecolor=black, urlcolor=black }

\usepackage{lipsum}

\begin{document}
\title{Nonlinear parameter-varying embeddings for nonlinear state estimation with
application to a two-link robot manipulator} %%%%%%%%%%%%
\author[Initial Surname]{JIAXIN JI, Shivaraj Mohite and JAN HEILAND}
%\address{Address}
%\email{example@mail.com}

% \let\thefootnote\relax
% \footnotetext{MSC2020: Primary 00A05, Secondary 00A66.} %%%%%%%%%%

\begin{abstract}
Observer design for nonlinear systems is a relevant and challenging task in
systems and control design. 
In this work, we follow the idea of embedding the system in the class of nonlinear
parameter-varying systems to benefit from linear structures as in a standard LPV
embedding while keeping some nonlinear structures and, thus, reducing the
numbers of scheduling-parameters in the representation. 
We lay out the NLPV observer design procedure for general nonlinear systems,
propose a number of improvements, and exemplify the application for a two-arm robot
model. 
In a numerical study, we compare the performance of the NLPV design to
established
standard nonlinear approaches such as the \emph{extended Kalman filter} and the
\emph{moving horizon estimation}.
\end{abstract} %%%%%%%%%
\maketitle
% \bigskip

% \noindent \lipsum[1] \cite{1}

% $\,$

% $\,$

%from section introduction to conclusion

\section{Introduction}

In this work, we consider nonlinear systems of type
\begin{equation}\label{eq:gen-nonl-uaffine}
  \dot x(t) = \bar f(x(t)) + B(x(t))\,u(t), \quad y(t) = Cx(t),
\end{equation}
and the general problem of observer design. That is we look for a dynamical
system that takes the available output $y$ of \eqref{eq:gen-nonl-uaffine} and
provides estimates on the otherwise unknown state $x$ to be used, e.g., for feedback control. 
For simplicity, we consider only affine control input and linear output.
The (non)linear parameter-varying (NLPV) theory, however,
readily extends to general nonlinear input and output
models. 
In what follows and where the context is clear, we will drop the
independent variables and write, e.g. $\rho$ or $g\,u$ rather than $\rho(x(t))$ or
$g(x(t))\,u(t)$.

For (N)LPV representations of \eqref{eq:gen-nonl-uaffine}, we seek parametrizations $\rho(x)$ and $\theta=\theta(x)$ of the system and the state
so that it is equivalently written as
\begin{equation}\label{eq:NLPV-factorization}
  \dot x = A(\theta)\,x + B(\theta)\, u + G(\theta)\,f(x), \quad y=Cx
\end{equation}
or
\begin{equation}\label{eq:LPV-factorization}
  \dot x = A(\rho)\,x + B(\rho)\, u , \quad y=Cx
\end{equation}
with some stretch in notation that $A$ and $B$ are different functions,
depending on which parametrization $\rho$ or $\theta$ is considered.
Here, the system \eqref{eq:LPV-factorization} is a classical LPV system, whereas
\eqref{eq:NLPV-factorization} represents a nonlinear LPV system that admits an
unparametrized nonlinearity $f$ for the sake of a simplified structure.

We note that such a linearly-varying parametrization is possible under mild
assumptions and is often indicated by the system's structure; see, e.g.,
\cite{BenH18} for a general procedure and \cite{HeiW23} on how an approximative
but lower-dimensional parametrization can be obtained. Since the NLPV
formulation is more general than the LPV, it exists under even less restrictive
conditions. We further note that (N)LPV embeddings are not unique in general.

% \todo[inline]{Looks like we don't do the LPV observer design at all...}
% \todo[inline]{Can we agree on using $\theta$ for NLPV representations and $\rho$
% for LPV?}

The main idea of the paper is to use recent developments on
observer design for NLPV systems, make it work for a quasi NLPV formulation of a
two-arm robot model as suggested in \cite{KoeT20}, and compare it to well established unstructured
methods for \eqref{eq:gen-nonl-uaffine} like the extended Kalman filter or
a \emph{moving horizon estimation} (MHE) as well as LPV designs for a
\eqref{eq:2arm-nonlin LPV form}. 
We note beforehand that for the used formulation, the plain LPV approaches as
proposed in \cite{ApkGB95} did
not work.
Indeed, the well established robust LPV controller design by \cite{ApkGB95} in
it's official Matlab implementation \texttt{hinfgs} was not able to provide a
feasible solution to the underlying system of linear matrix equations (LMI).
We interpret this negative result in favor of the flexibility offered by the NLPV design.

Further suggested designs in the literature are provided in \cite{AlJZ17} that proposes
an LMI based state-observer for polytopic LPV systems that is robust to
uncertainties in the parameter and in
    \cite{Mag05} which is a standard Luenberger observer for a linear fractional
    representation of an LPV system. Since in our model, the considered nonlinearities are not
    of rational type, a conversion into LFT form is not applicable.

Concerning the class of nonlinear parameter-varying systems, we refer to recent
publications from control system engineering; see \cite{RAJAMANI_NLPV_Obs,2019_Discrete_Khadilja,LPV_2020_RNC}. 
Applications have been reported in
% A nonlinear parameter-varying (NLPV) framework is employed for the 
automotive suspension models (see~\cite{LPV_2020_RNC, tran2022unified}) and
lateral vehicle dynamics; see~\cite{RAJAMANI_NLPV_Obs}. 
% A previous treatment of a two-arm robot model, apparently, has not been reported so far. 

In this work, we extend the LMI-based controller design for NLPV
systems, as we developed it in \cite{CDC_2024_Shiv_GDO,Shiv_LCSS_CDC_2025},
towards controller design. 

Apart from reducing the overall number of parameters (if compared to LPV representations),
the NLPV approach provides the freedom to only choose scheduling variables for
the observer that are directly observable. Thus, the practical and theoretical problem of
the possible need for estimating scheduling parameter from the measured outputs
is avoided. For references to the treatment of unknown parameters, see the works by \cite{AlJZ17}
(a general $H_\infty$ approach) or \cite{HalMD13} (for discrete-time systems).
% and \cite{BarDK01}(cont-time) ?) is used as a benchmark.

Throughout the paper, the ensuing notations are deployed:
% \begin{enumerate}[a)]
    % \item The initial value of $e(t)$ at $t=0$ is represented by $e_0$.   
      The symbols $||e(t)||$ and $||e||_{\mathcal{L}_2}$ depict the euclidean
      norm of a function value (i.e., a vector) and the $\mathcal{L}_2$ norm of
      a vector-valued function $e$, respectively. 
    % \item 
      The identity matrix and the null matrix are described by $\mathbb{I}$ and
      $\mathbb{O}$, respectively. Where it helps the understanding, we add a
      subscript like $\mathbb I_q$ to express the size of the matrices.
    % \item 
      The transpose of a matrix $A$ is illustrated as  $A^\top$. The notation $A \in \mathbf{S}^{n}$ indicates that $A\in \mathbb{R}^{n \times n}$ is a symmetric matrix. The symbol $(\star)$ denotes repeated blocks within a symmetric matrix. Further, the minimal and maximal eigenvalues of this matrix $A$ are expressed as $\lambda_{\min}(A) $ and $\lambda_{\max}(A)$, respectively. 
    %\item The term ${\| A \|}_2 = \sqrt{\lambda_{\max}(A^\top \cdot A)}$ depicts the Euclidean norm of the matrix $A\in \mathbb{R}^{m \times n}$ and it is computed by using ${| A |}2 = \sqrt{\lambda{\max}(A^\top \cdot A)}$. 
    % \item  
      $A = \text{block-diag}(A_1, \hdots,A_n)$ denotes a block-diagonal matrix having elements $A_1, \hdots, A_n$ in the diagonal. 
    % \item 
      $A\in \mathbf{S}^n_{+}$ and $A\in \mathbf{S}^n_{++}$ are used to represent a positive semidefinite ($A\geq 0$) and positive definite ($A>0$), respectively. 
    % \item 
      Similarly, the inequalities $A<0$ and $A\leq 0$ indicate that $A \in \mathbf{S}^{n}$ is negative definite and negative semidefinite, respectively.
    % \item 
      The sum of the matrix $X$ and its transpose is described as a function $\mathcal{H_E}(X)$, i.e., $\mathcal{H_E}(X)=X^\top+X$.
\section{NLPV observer design}
% In the previous section, the authors have described the classical approach for the state estimation of the LPV model of a two-link robotic arm manipulator~\eqref{eq:2arm-nonlin LPV form}. 
% In recent times, the synthesis of observers for Nonlinear Parameter-Varying (NLPV) systems has become a popular research topic in control system engineering (\cite{2019_Discrete_Khadilja,LPV_2020_RNC. In this section, the authors will focus on the observer-based state estimation of system~\eqref{eq:2arm-nonlin NLPV form}. 
  % Further, we will proposed an LMI-based methodology to compute the observer parameters.

% \subsection{Introducing NLPV observer}
We recall the general idea of NLPV observer design.
Let us consider the following generalised disturbance-affected NLPV model:
\begin{equation}\label{gen NLPV form with d}
    \begin{split}
        \dot{x}(t)&= A(\theta) x(t)+ G(\theta) f(x(t))+B(\theta) u(t)+E
        \omega(t),\\
        y(t)&=C x(t)+D \omega,
    \end{split}
\end{equation}
where $x(t)\in \mathbb{R}^{n}$, $u(t)\in \mathbb{R}^s$ and $y(t)\in \mathbb{R}^p$
represent the states of the systems, control input, and the system's output,
respectively. $\omega(t) \in \mathbb{R}^{q}$ depicts the bounded disturbance/noise
present inside the system dynamics and output. The matrices
$A(\cdot),~B(\cdot),~G(\cdot)$ are functions of the vector $\theta(t)$. ~$C,~E$
and $D$ are constant matrices of appropriate dimension. The $f\colon \mathbb
R^{n}\to
\mathbb{R}^m$ represents the nonlinear function inside the system dynamics for
which we assume boundedness and sufficient smoothness so that
\begin{equation}\label{gen NLPV form f lip bound}
    \bar{f}_{ij_{\min}}\leq \dfrac{\partial f_i(x)}{\partial x_j}\leq
    \bar{f}_{ij_{\max}},\quad \text{for }i=1,\dotsc,m,\quad j=1,\dotsc,n. 
\end{equation} 
%It leads to
%\begin{equation}\label{gen NLPV form f Gamma bound}
%       \underline{\Gamma} \leq \Gamma\leq \bar{\Gamma,}
%    \end{equation}
%    where $\underline{\Gamma}\triangleq \begin{bmatrix}
%        f_{ij_{\min}} 
%    \end{bmatrix}_{m\times n}$ and $\bar{\Gamma}\triangleq \begin{bmatrix}
%        f_{ij_{\max}} 
 %   \end{bmatrix}_{m\times n}$.

% \todo[inline]{JH: this should be reformulated. (1) a Lipshitz function may not
%   have a Jacobian matrix (2) $\Gamma$ is not symmetric (and probably not
%   definite) so that $\geq$ can not be used.}

Concerning the time-varying parameter $\theta(t)\in \mathbb{R}^{{n_\theta}}$ and
the coefficients in the system~\eqref{gen NLPV form with d}, we impose the
following restrictions:
\begin{assumption}~
  \begin{itemize}
    \item[(A1)] The time varying parameter $\theta(t)$ is assumed to be known and bounded, i.e., it holds:
\begin{equation}\label{gen NLPV form theta bound}
   \theta_{i_{\min}}   \leq \theta_{i}\leq \theta_{i_{\max}}, \quad \forall {i}=\{1,\hdots,{n_\theta}\}. 
\end{equation}
\item[(A2)] The system matrices $A(\cdot)\in \mathbb{R}^{n \times n}$, $B(\cdot)\in
\mathbb{R}^{n \times s}$, and $G(\cdot)\in \mathbb{R}^{n \times m}$ are known
functions that are \emph{affine} in the parameter $\theta(t)$, i.e.
% Since $\theta(t)$ is known, 
one can rewrite the matrices $A(\theta),~B(\theta),~G(\theta)$ as:
\begin{equation*}%\label{gen NLPV form sys matrcies}
    \begin{split}
      A(\theta(t))=A_0+{\sum^{n_\theta}_{i=1}}\theta_i A_i,~ \\
      B(\theta(t))=B_0+{\sum^{n_\theta}_{i=1}}\theta_i B_i,~ \\
       G(\theta(t))=G_0+{\sum^{n_\theta}_{i=1}}\theta_i G_i. \
    \end{split}
\end{equation*}
with given constant matrices $A_i$, $B_i$, $G_i$, for $i=0,1, \dotsc, n_\theta$.
\item[(A3)] The matrices $C\in \mathbb{R}^{p \times n},~E\in \mathbb{R}^{n \times q}$ and $D\in \mathbb{R}^{p \times q}$ are constant matrices.
\end{itemize}
\end{assumption}
\begin{remark}\label{Rem 1 noise}
    If the system dynamics and output are affected by two distinct noise vectors, $\omega_1$ and $\omega_2$, through $E_1$ and $D_1$ respectively, then the system can be rewritten in the form of~\eqref{gen NLPV form with d} by introducing $$E = \begin{bmatrix}E_1 & \mathbb{O}\end{bmatrix},~D = \begin{bmatrix}\mathbb{O} & D_1\end{bmatrix},~\text{and}~\omega = \begin{bmatrix}\omega_1^\top & \omega_2^\top\end{bmatrix}^\top.$$
\end{remark}
The subsequent Luenberger-like observer is employed for the state estimation of the system~\eqref{gen NLPV form with d}:
\begin{equation}\label{NLPV obs}
    \dot{\hat{x}}(t)= A(\theta) \hat{x}(t)+ G(\theta) f(\hat{x}(t))+B(\theta) u(t)+L(\theta)(y-C \hat{x}(t)),
\end{equation}
where $\hat{x}(t)\in \mathbb{R}^n$ is an estimated states. The observer gain $L(\theta)$ is expressed in the following form:
\begin{equation}\label{NLPV obs L form}
    \begin{split}
      L(\theta(t))&=L_0+{\sum^{n_\theta}_{i=1}}\theta_i L_i.
    \end{split}
\end{equation}

Let us define the estimation error as $e(t)=x(t)-\hat{x}(t)$. Through the use of~\eqref{gen NLPV form with d} and~\eqref{NLPV obs L form}, one can deduce:
\begin{equation}\label{NLPV e dyn}
    \begin{split}
     \dot{e}(t)&= (A(\theta) - L(\theta) C)\hat{e}(t)+ G \tilde{f}+(E-L(\theta)D)\omega,
    \end{split}
\end{equation}
where $\tilde{f}=f(x(t))-f(\hat{x}(t))$.

In what follows, we present a newly developed LMI condition to determine the observer gain $L(\theta)$, such that
\begin{enumerate}[1)]
    \item When $\omega=0$, the estimation error dynamic~\eqref{NLPV e dyn} is exponentially stable.
    \item When $\omega\neq 0$, the  dynamic~\eqref{NLPV e dyn} fulfill the ensuing
$\mathcal{H}_\infty$ criterion:
\begin{equation}\label{NLPV e H inf criterion}
    \|\tilde{x}\|_{\mathcal{L}_2}\leq \sqrt{\nu_1 \|\tilde{x}_0\|^2+\nu_2\|\omega\|^2_{\mathcal{L}_2}},
\end{equation}
where $\nu_1,~\nu_2 >0$.
\end{enumerate}
 
%In order to compute the parameters of the observer~\eqref{NLPV obs}, we state the following theorem:
\section{Performance of NLPV Observer}
%For the enhancement of comprehensibility, the authors have divided the proof into two parts.\begin{itemize}\item[A]  
Let us consider the subsequent Lyapunov function for the stability analysis:
    \begin{equation}\label{NLPV edot Ly}
        V(e(t))= (e(t))^\top P e(t),~P\in \mathbb{S}^n_{++} .
    \end{equation}
    It yields:
    \begin{equation*}
    \lambda_{\min}(P)\|e(t)\|^2_{2} \leq V(e(t)) \leq \lambda_{\max}(P)\|e(t)\|^2_{2}.
    \end{equation*}
According to~\cite[Remark 1]{Shiv_RNC_2024_LMI}, the error dynamics~\eqref{NLPV e dyn} satisfy $\mathcal{H}_\infty$ criterion~\eqref{NLPV e H inf criterion} if Lyapunov function~\eqref{NLPV edot Ly} fulfils:
\begin{equation}\label{NLPV e H inf con}
    \dot{V}(e(t))+\sigma V(e(t))-\mu \omega^\top\omega \leq 0,
\end{equation}
where $\sigma,~\mu$ are positive scalars.
    
Further,
    \begin{equation*}
    \begin{split}% \label{NLPV edot Ly vdot}
        \dot{V}(e(t))&=(e(t))^\top \big( \mathcal{H_E}(P(A(\theta) - L(\theta) C))\big)e(t)
        \\
        &+\mathcal{H_E}\big((e(t))^\top P (E-L(\theta)D)\omega\big)+\mathcal{H_E}\big((e(t))^\top P G \tilde{f}\big) .
    \end{split}
    \end{equation*}
    Through the implementation of Young inequality~\cite[eqn. (1)]{CDC_2024_Shiv_GDO}, we deduce:
    \begin{equation*}
        \begin{split}%\label{NLPV edot Ly vdot young 1}
            \mathcal{H_E}\big((e(t))^\top P &G \tilde{f}\big)\leq \tilde{f}^\top M \tilde{f}+(e(t))^\top \big( P G (M^{-1})G^\top P\big)e(t),
        \end{split}
    \end{equation*}
    where $M\in \mathbb{S}^m_{++}$.\\
   % From~\eqref{gen NLPV form f lip},
   % \begin{equation}\label{NLPV edot Ly vdot young 1 bound}
  %        \tilde{f}^\top M \tilde{f}\leq (e(t))^\top \big( \Gamma^\top (M) \Gamma \big)e(t),
  %  \end{equation}
    By using~\eqref{gen NLPV form f lip bound}, we get:
    \begin{equation*}%\label{NLPV edot Ly vdot young 1 bound 1}
          \tilde{f}^\top M \tilde{f}\leq (e(t))^\top \big( \bar{\Gamma}^\top (M) \bar{\Gamma} \big)e(t),
    \end{equation*}
    where $\bar{\Gamma}\triangleq \begin{bmatrix}
        f_{ij_{\max}} 
    \end{bmatrix}_{m\times n}.$\\
    Thus, one can deduce:
    \begin{equation}
    \begin{split}\label{NLPV edot Ly vdot bound}
        \dot{V}(e)\leq \begin{bmatrix}
           e\\  \omega
        \end{bmatrix}^\top \begin{bmatrix}
           \Pi_{11}&P (E-L(\theta)D)\\\star &\mathbb{O}
        \end{bmatrix}\begin{bmatrix}
           e\\  \omega
        \end{bmatrix},
    \end{split}
    \end{equation}
where $\Pi_{11}=\mathcal{H_E}(P(A(\theta) - L(\theta) C))+ \big( \bar{\Gamma}^\top (M) \bar{\Gamma} \big)+\big( P G (M^{-1})G^\top P\big)$.

Through the use of~\eqref{NLPV edot Ly vdot bound}, the condition specified in~\eqref{NLPV e H inf con} is fulfilled if
\begin{equation}\label{NLPV e H inf con nes}
    \begin{bmatrix}
           \Pi_{11}
        +\sigma P&P (E-L(\theta)D)\\\star &-\mu \mathbb{I}_q 
        \end{bmatrix}    \leq 0.
\end{equation}
By deploying the Schur Lemma stated in~\cite[Sec. 2.3]{Schur_ref_2019}, we deduce:
\begin{equation}\label{NLPV e H inf con nes schur}
    \begin{bmatrix}
          \Sigma_1&PE-R(\theta) D& P G\\\star &-\mu \mathbb{I}_q &\mathbb{O}\\
        \star&\star&-M
        \end{bmatrix}    \leq 0,
\end{equation}
where  $\Sigma_1=\mathcal{H_E}(P A(\theta))-\mathcal{H_E}(R(\theta) C)+\bar{\Gamma}^\top (M) \bar{\Gamma}
        +\sigma P$ and $R(\theta)=P L(\theta)$. 
From~\eqref{NLPV obs L form}, we can deduce: $R(\theta)=R_0+{\sum^{n_\theta}_{i=1}}\theta_i R_i$, along with $R_j=P L_j,~j=0,1,\hdots,n_\theta$. 

Now, the condition specified in~\eqref{gen NLPV form theta bound} implies that all the elements of $\theta$ are bounded and belong to a convex set $\mathcal{V}_{\theta}$, for which the sets of vertices are given by:
\begin{equation}\nonumber
    \begin{split}
     \mathcal{V}_{\theta}&=\big\{
\{\theta_{1},
\hdots,\theta_{n_\theta}\} : \theta_{i} \in [\theta_{i_{\min}} ,\theta_{i_{\max}}]
  \big\} , 
    \end{split}
\end{equation}
Hence, the inequality~\eqref{NLPV e H inf con nes schur} is fulfilled if
\begin{equation*}%\label{NLPV e H inf con nes schur con}
    \begin{bmatrix}
          \Sigma_1&PE-R(\theta) D& P G\\\star &-\mu \mathbb{I}_q &\mathbb{O}\\
        \star&\star&-M
        \end{bmatrix}_{\theta \in \mathcal{V}_\theta}    \leq 0,
\end{equation*}

\begin{theorem}
Let us consider the matrices~$P\in \mathbb{S}^n_{++}$,~$Z\in \mathbb{S}^m_{++}$ and $R_i\in \mathbb{R}^{n\times p},~i=0,1,\hdots,n_\theta$, and positive scalars $\mu,~\sigma$.
If the following optimisation problem is solvable:
\begin{equation}\label{NLPV e H inf LMI}
    \begin{split}
        \min\mu\text{ subject to}&\\
        \begin{bmatrix}
          \Sigma_1&PE-R(\theta) D& P G\\\star &-\mu \mathbb{I}_q &\mathbb{O}\\
        \star&\star&-M
        \end{bmatrix}_{\theta \in \mathcal{V}_{\theta} }  & \leq 0,
    \end{split}
\end{equation}
where $\mathcal{H_E}(P A(\theta))-\mathcal{H_E}(R(\theta) C)+\bar{\Gamma}^\top (M) \bar{\Gamma}
        +\sigma P$
then, the error dynamics~\eqref{NLPV e dyn} satisfy $\mathcal{H}_\infty$ criterion~\eqref{NLPV e H inf criterion}. The observer parameters are computed by using $L_j=P^{-1} R_j,~j=0,1,\hdots,n_\theta$.
\end{theorem}
\begin{proof}
From the convexity principle~(\cite{boyd1994linear}), if LMI~\eqref{NLPV e H inf LMI} is solved for all $\theta \in \mathcal{V}_{\theta}$, then one can ensure that the condition specified in~\eqref{NLPV e H inf con nes} is satisfied. Since the Lyapunov function~\eqref{NLPV edot Ly} of the system~\eqref{NLPV e dyn} fulfills~\eqref{NLPV e H inf con nes}, the trajectories of error dynamics~\eqref{NLPV e dyn} converge asymptotically to zero with optimal noise attenuation level $\mu$. Thus, the error dynamic~\eqref{NLPV e dyn} holds:$$\|e(t)\|^2_{\mathcal{L}_2} \leq \mu \|\omega\|^2_{\mathcal{L}_2}$$
\end{proof}
%-------------------------------------------------
% In the sequel, we have deployed the EKF methodology and the proposed LMI-based observer method for state estimation of a two-link robotic arm manipulator~\eqref{2arm-nonlin model}.
% 

\section{LPV and NLPV Represenations of a Two-link Arm Robot Model}
% \todo[inline]{Consolidate notation -- would be good to have $A$, $B$, $C$, $D$
%   for the (NLPV) system and $f$ and $g$ for the nonlinear control affine system.
%   Thus, we should probably rename the coefficients and constants of the
%   model. Anyways, JH finds that $n$ and $f$ are somehow unexpected here. But
%   maybe there are conventions or reasonable choices for naming these model
% parameters.}
In this section, we exemplify how a nonlinear model can be described in terms of
(quasi) LPV and NLPV equations.

A two-link arm robot manipulator can be described using the following equations of motion
\begin{equation}\label{2arm-nonlin model}
M(q(t))\ddot{q}(t) + c(q(t), \dot{q}(t))\dot{q}(t) + g(q(t)) = p_7\,\tau(t),
\end{equation}
where $q(t) := (q_1(t), q_2(t))$ are the angles and
$\tau(t):=(\tau_1(t),\tau_2(t))$ represents the motor torques. The model functions $M$, $c$, and $g$ are defined as
% \begin{equation*}
%   \begin{split}
%     M(q(t)) &= \begin{bmatrix}
%     p_1 & p_2 \cos(q_\Delta (t)) \\ 
%     p_2 \cos(q_\Delta (t)) & p_3
% \end{bmatrix}, \\
%       g(q(t)) &= \begin{bmatrix}
%     -p_4 \sin(q_1(t)) \\ 
%     -p_5 \sin(q_2(t))
% \end{bmatrix}, \\
%         c(q(t), \dot q(t)) &= \begin{bmatrix}
%     p_2 \sin(q_\Delta (t)) \dot{q}_2(t)^2 + p_6   \dot{q}_1(t) \\ 
%     -p_2 \sin(q_\Delta(t)) \dot{q}_1(t)^2 + p_6 (\dot{q}_2(t) - \dot{q}_1(t))
% \end{bmatrix}.
%         \end{split}
% \end{equation*}

\begin{equation*}
\begin{aligned}
M(q(t))
&=
\begin{bmatrix}
p_1 & p_2\cos(q_\Delta(t)) \\
p_2\cos(q_\Delta(t)) & p_3
\end{bmatrix},
\\
g(q(t))
&=
\begin{bmatrix}
-p_4\sin(q_1(t)) \\
-p_5\sin(q_2(t))
\end{bmatrix},
\\
c(q(t),\dot q(t))
&=
\begin{bmatrix}
p_2\sin(q_\Delta(t))\dot q_2(t)^2+p_6\dot q_1(t) \\
-p_2\sin(q_\Delta(t))\dot q_1(t)^2
+p_6\bigl(\dot q_2(t)-\dot q_1(t)\bigr)
\end{bmatrix}.
\end{aligned}
%\label{eq:robot-dynamics-components}
\end{equation*}

with $q_\Delta(t) := q_1(t) - q_2(t)$.
The values of the parameter vector $p$ of the model are given in
Table~\ref{tab:params-twoarm-robot}. 

\begin{table*}[t]
\centering
\caption{Model parameters of the two-link robot arm.}
\label{tab:params-twoarm-robot}
\begin{tabular}{|c|c|c|c|c|c|c|c|}
\hline
\textbf{Parameter:} & $p_1$ & $p_2$ & $p_3$ & $p_4$ & $p_5$ & $p_6$ & $p_7$ \\ \hline
\textbf{Value: (in the simulation)} & 5.6794 & 1.473 & 1.7985 & 0.4 & 0.4 & 2 & 1 \\ \hline
\textbf{Physical interpretation:} & 
\multicolumn{3}{|c|}{{Moments of inertia}} & 
\multicolumn{2}{|c|}{{Gravity scaling}} & 
{Friction} & 
{Input scaling} \\ \hline
\end{tabular}
\end{table*}

Taking the torques as controls, we rewrite the model as a nonlinear
first-order state-space equation
\begin{equation}\label{eq:2arm-nonlin}
\begin{split}
  \dot{x}(t) &= \bar f(x(t)) + \bar B(x(t))\, u(t) \\
    \quad y(t) &= 
    \begin{bmatrix}
        \mathbb I_2 & \mathbb O
    \end{bmatrix} x(t),
    \end{split}
\end{equation}
with
\begin{equation*}%\label{eq:2arm-nonlin-coefficients}
  \begin{split}
    \bar f(x) &= \begin{bmatrix}
      \begin{bmatrix}
    x_3(t) \\ x_4(t)
\end{bmatrix}  \\
        M(\begin{bmatrix}
    x_1(t) \\ x_2(t)
\end{bmatrix})^{-1}g(\begin{bmatrix}
    x_1(t) \\ x_2(t)
\end{bmatrix})- M(\begin{bmatrix}
    x_1(t) \\ x_2(t)
\end{bmatrix}))^{-1}c(\begin{bmatrix}
    x_1(t) \\ x_2(t)
\end{bmatrix}, \begin{bmatrix}
    x_3(t) \\ x_4(t)
\end{bmatrix})\,
        \end{bmatrix},\\
      \bar B(x) &= \begin{bmatrix}
        0 \\ p_7 M(\begin{bmatrix}
    x_1(t) \\ x_2(t)
\end{bmatrix})^{-1}
      \end{bmatrix}\\
  \end{split}
\end{equation*}
and where $x(t) = (\begin{bmatrix}
    x_1(t) \\ x_2(t)
\end{bmatrix}, \begin{bmatrix}
    x_3(t) \\ x_4(t)
\end{bmatrix}) := (q(t), \dot{q}(t))$ and $u(t) =
\tau(t)$. In the LPV notation that follows, we will consider $x(t)$ as a column
vector, i.e.
\begin{equation}\label{eq:xascvec}
  x(t) = \begin{bmatrix}
   x_1(t)\\x_2(t)\\x_3(t)\\x_4(t)
\end{bmatrix}=\begin{bmatrix}
   q_1(t)\\ {q}_2(t)\\\dot{q}_1(t)\\\dot{q}_2(t)
\end{bmatrix}.
\end{equation}

\subsection{Quasi LPV representation}
We consider a so-called \emph{scheduling map} $\rho(t)=\rho(x(t))=\rho((q(t),\dot q(t)))$ chosen
as in \cite{KoeT20} as $\rho: \mathbb{R}^4 \to \mathbb{R}^{10}$:
\begin{equation}\label{eq:lpv-rho-sched-map}
  % \begin{split}
    \rho((q,\dot{q})) = % \\
            \begin{bmatrix}
        \frac 1h \\
        \frac 1h \cos(q_\Delta) \\
        \frac 1h \sin(q_1)\frac{1}{q_1} \\
        \frac 1h \cos(q_\Delta)\sin(q_2)\frac{1}{q_2} \\
        \frac 1h (-p_2^2 \sin(q_\Delta) \cos (q_\Delta) \dot q_1 - p_6(p_3 +
        p_2\cos(q_\Delta))) \\
        \frac 1h (-p_3\sin(q_\Delta) \dot{q_2} +  p_6\cos(q_\Delta)) \\
        \frac 1h \cos(q_\Delta) \sin(q_1)\frac{1}{q_1} \\
        \frac 1h \sin(q_2)\frac{1}{q_2} \\
        \frac 1h (p_1p_2\sin(q_\Delta)\dot{q_1} + p_6(p_1 + p_2\cos(q_\Delta))) \\
      \frac 1h (p_2^2\sin(q_\Delta)\cos(q_\Delta) \dot{q}_2 - p_1p_6)
    \end{bmatrix},
  % \end{split}
\end{equation}
where $h = \operatorname{det}(M) = p_1p_3 - p_2^2\cos^2(q_\Delta)$.
We note that the function $\frac{\sin(s)}{s}$ is continiously extendable at
$s=0$ and that with the chosen parameters (see Tab.
\ref{tab:params-twoarm-robot}) the value of $1/h$ is bounded.
With the help of \eqref{eq:lpv-rho-sched-map}, the nonlinear model \eqref{eq:2arm-nonlin} with $x$ as defined in
\eqref{eq:xascvec} can be equivalently rewritten as a so-called quasi LPV system:
\begin{equation}\label{eq:2arm-nonlin LPV form}
\begin{aligned}
    \dot{x}(t) &= A(\rho(t))x(t) + B(\rho(t))u(t), \\
    y(t) &= Cx(t) + Du(t),
    \end{aligned}
\end{equation}
with
\begin{equation*}
  \begin{split}
    A(\rho) &= 
    \begin{bmatrix}
        0 & 0 & 1 & 0 \\
        0 & 0 & 0 & 1 \\
       - p_3 \, p_4\,\rho_3 & p_2 \, p_5\,\rho_4 & \rho_5 & p_2\,\rho_6 \\
        p_2 \, p_4\,\rho_7 & - p_1\, p_5\,\rho_8 & \rho_9 & \rho_{10}
    \end{bmatrix},\\
    B(\rho) &= 
    \begin{bmatrix}
        0 & 0 \\
        0 & 0 \\
        p_3\, p_7\,\rho_1 & -p_2\, p_7\,\rho_2 \\
        -p_2\, p_7\,\rho_2 & p_1\, p_7\,\rho_1
    \end{bmatrix}, \quad
    C = \begin{bmatrix}
       \mathbb  I_2 & \mathbb O
    \end{bmatrix}, \quad D = \mathbb{O}.
  \end{split}
\end{equation*}
Here all nonlinearities are factorized into the coefficient matrices $A$ and
$B$.

%a link to connect linear and nonlinear models Maybe we can use example of vehicle suspension 
% ~\textcolor{red}{Is it possible to add few comments related to limitations with the linear model?}

% ~\textcolor{blue}{NLPV model is updated with new sign change}

%%%%%%%%%%Shivaraj...%%%%%%%%%%
\subsection{Quasi NLPV representation}
% In this subsection, we derive yet another representation of the aforementioned
% two-link robotic arm model as nonlinear LPV system. 

A LPV representation as in \eqref{eq:2arm-nonlin LPV form} seems
beneficial
because of it immediate alignment with LPV theory. However, a complicated and
large-dimensional scheduling map can make standard LPV observer design
difficult because of the involved dependencies and possibly too conservative
bounds on the parameter, that makes the design infeasible.

The idea of NLPV tries to mitigate this effect by factorizing only a part of
the dynamics and treating some parts explicitly as nonlinearity.

By straight-forward comparison arguments, we find that
the LPV system~\eqref{eq:2arm-nonlin LPV form} can be written in the following form:
\begin{equation}\label{eq:2arm-nonlin NLPV form}
    \begin{split}
        \dot{x}(t)&= A(\theta(t)) x(t)+ G(\theta(t)) f(x(t))+B(\theta(t)) u(t),\\
        y(t)&=C x(t),
    \end{split}
\end{equation}
where 
\begin{equation}\nonumber
    \begin{aligned}
     A(\theta(t))  &= \begin{bmatrix}
          0&0&1&0\\0&0&0&1\\
          0&0&\frac{1}{h}\left(- (p_3 + p_2\theta_1) p_6 \right)&\frac{1}{h}\left(\theta_1 p_6 \right)\\
          0&0&\frac{1}{h}\left( (p_1+p_2\theta_1) p_6 \right) &\frac{-p_1 p_6}{h}
\end{bmatrix},
\\
~G(\theta(t))&=\begin{bmatrix}
0&0&0&0\\ 0&0&0&0\\ -\frac{p_3 p_4}{h} &+\frac{p_2 p_5}{h}(\theta_1)&-\frac{p_2^2}{h} \theta_3 &-\frac{p_3 p_2}{h}\theta_2 \\\frac{p_2 p_4}{h}(\theta_1)&-\frac{p_1 p_5}{h} &\frac{p_1 p_2}{h}\theta_2 & \frac{p_2^2}{h}\theta_3 
\end{bmatrix},\\ 
f(x(t))&=\begin{bmatrix}
\sin(x_1)\\\sin(x_2)\\x^2_3\\x^2_4
      \end{bmatrix},
B(\theta(t))=
    \begin{bmatrix}
          0&0\\
          0&0\\
          \frac{p_3}{h}&\frac{-p_2}{h} \theta_1\\
          \frac{-p_2}{h} \theta_1&\frac{p_3}{h}
\end{bmatrix}
    \end{aligned}
\end{equation}
along with $\theta_1(t)=\cos(x_1(t)-x_2(t))$, $\theta_2(t)=\sin(x_1(t)-x_2(t))$,
$\theta_3(t)=\theta_1(t)\cdot \theta_2(t)$ and $h(t)=ac - b^2\theta^2_1(t)$.

System~\eqref{eq:2arm-nonlin NLPV form} encompasses fewer state-dependent
parameters but incorporates additional nonlinearities compared
to~\eqref{eq:2arm-nonlin LPV form}. Thus, the model identification task is much
easier in the case of the system~\eqref{eq:2arm-nonlin NLPV form}. Since the
time-varying parameters $\theta_1(t),~\theta_2(t)$ and $h(t)$ are depending on
the state vector $x(t)$, the model~\eqref{eq:2arm-nonlin NLPV form} is commonly
referred to as~\emph{quasi-NLPV} model, 
where the term \emph{quasi} points to the deviation to a standard (N)LPV model,
where the parameter is an external variable.

Another advantage is that $\theta_1,~\theta_2,~\theta_3$ and $h$
only depend on measured output $y(t)$. Thus, we can consider $\theta_1(x(t))$,
$\theta_2(x(t))$, $\theta_3(x(t))$ and $h(x(t))$ as known during the simulation. 
With that, we can also assume the we can find empirical bounds  on
$\theta_1(t),~\theta_2(t),~\theta_3(t)$ and $h(t)$, i.e.,
\begin{equation*}
    \begin{split}
     \theta_{1_{\min}}  \leq \theta_1(t)\leq   \theta_{1_{\max}},
     \theta_{2_{\min}}  \leq \theta_2(t)\leq   \theta_{2_{\max}}, \\
     \theta_{3_{\min}}  \leq \theta_3(t)\leq   \theta_{3_{\max}}, 
     h_{{\min}}  \leq h(t)\leq   h_{{\max}}.
    \end{split}
\end{equation*}
%reviwer might be not agree if we directly start with Quasi LPV or NLPV observer we need to add certain remark and reason to treat system as LPV or NLPV instead of Quasi NLPV and we can mention Quasi NLPV if the considered assumption is not hold

\section{Numerical Experiments}

In the sequel, we will present numerical nonlinear state estimation results for 
two-link robot manipulator~\eqref{2arm-nonlin model} for the presented NLPV
approach and compare to the established \emph{extended Kalman filter} and to the
approach of \emph{moving horizon estimation} (MHE), whose initial concepts were presented in a prior study \cite{Ji24}.

For the simulation, the true initial state was \(x(0)=[1,1,1,1]^\top\), whereas
the observers were initialized at $\hat x(0)=[0,0,0,0]^\top$. 

% \todo[inline]{JH: Question: Generally spoken, if we do not know the state $x(t)$,
%   we don't have the $\rho(x(t))$. Practically, we will use the estimate
% $\rho(\tilde x)$. Do we need extra analysis for that? Is this the same as
% \emph{robust} observer design? Check \cite{Sat12}}
% 
% Compare also the \emph{correct} $\rho(x)$ with the estimated $\rho(\tilde x)$.
% 
% Plan for Numerical Experiments
% 
% 
% \begin{itemize}
%   \item quantity to observe: $x_2$ which is $\dot q$
%   \item try/compare with EKF -- do we need noise for that?
%   \item and/or MHE? -- do we need noise for that?
%   \item then: qLPV and qNLPV
% \end{itemize}

%EKF results------------------------------------------------------
%-------------------------------------------------------
%------------------------------------------------------
\subsection{State Estimation with the Extended Kalman Filter}
    As a benchmark, we evaluated the performance of the standard continuous‐time \emph{extended
    Kalman Filter} (EKF) \cite{Rib04} on the two‐link robot arm. The simulations were carried
    out with a fixed step size of $\Delta t = 0.001$ over a total duration of $T
    = 15$ ($N = 15000$) steps. Both measurement noise $v$ and state noise $w$ are assumed
    to be zero-mean Gaussian with identical covariance, i.e.
    \begin{equation*}
    v(t) \sim \mathcal{N}(0, \,0.1\cdot \mathbb{I}_{4 }), \quad w(t) \sim
    \mathcal{N}(0, \,
    0.1\cdot \mathbb I_{2 } ). %\mathbb I_2
   \end{equation*}
The initial estimation covariance was set to $P(0)=\mathbb I_{4}$. Figure~\ref{fig 1 state vs time NLPV obs} shows the true and estimated trajectories of the four states. We quantify the average error per step by
\begin{equation*}
  {e_{ekf}}_i 
  = \| x_i - \hat{x}_i \|_{\mathcal{L}_2}, \quad\text{for }
% = \sqrt{\frac{1}{N}\sum_{k=1}^{N} \bigl(x_i(k) - \hat x_i(k)\bigr)^2},
\quad i=1,\dots,4.
\end{equation*}
where the norm is evaluated via a numerical integration over the data given at
the discrete time instances.
% \todo[inline]{Write error with the $\mathcal L_2$ norm.}
For the given noise levels, the resulting estimation error values for duration of $10\leq t \leq 15$ were illustrated in Table~\ref{tab 2 rmse}.

%figure 
%\begin{figure}[ht]
%  \centering
%  \includegraphics[width=\linewidth]{image/ekf.pdf} 
%  \caption{True and estimated state trajectories under EKF}
% \label{fig:ekf}
%\end{figure}
%Can you make figure wrt time instead of sample?
It is shown that the EKF estimation exhibits a high degree of consistency with the true state trajectories.
%MHE-------------------------------------------------------------
\subsection{State Estimation with the Moving Horizon Estimation}
For comparison, we implemented a MHE approach with fixed horizon length $
(N_k=10)$ and the same noise as used for the EKF simulation.
Following the MHE design \cite{RawMD17}, at each sampling time $k$, we consider
only the historical data from $k-N_k$ to $k-1$, and for $k<N_k$, generally
consider the MHE problem to be the problem commonly classified as \emph{full
information estimation} (FIE):
\begin{equation*}
\begin{aligned}
V_k(\chi_0, \boldsymbol{\omega})  :=& \min_{x_0, \omega, \nu} \, \ell_0(\chi_0 - \bar{x}_0) 
+ \sum_{i=0}^{k-1} \ell(\omega_{i|k}, \nu_{i|k}) \\
\text{subject to } \quad 
& \chi_{i+1|k} = f(\chi_{i|k}, \mu_{i|k}) + \omega_{i|k} \\
& y_{i|k} = h(\chi_{i|k}, \mu_{i|k}) + \nu_{i|k}
\end{aligned}
\end{equation*}
where $\bar{x}_0=[0 \quad 0 \quad 0 \quad 0]^\top$ is the prior information on the initial state, $\mathbf{y}_k = \begin{bmatrix} y_{0|k}^T, \dots, y_{k|k}^T \end{bmatrix}^T$ is a sequence of measurements from moment $0$ to $k - 1$, the prior weighting matrix and stage cost matrix are defined as: $\ell_0(\cdot)=\mathbb I_4$ and $\ell(\cdot)=200\cdot \mathbb I_4$. The decision state estimates, state and measurement disturbance estimates are denoted as $\boldsymbol{\chi} $, $\boldsymbol{\omega}$ and $\boldsymbol \nu $, respectively.

%For $j \in \mathbb{I}_{0:k}$, we refer to the $j^{\text{th}}$ element of
%$\hat{\mathbf{x}}_k$ as $\hat{x}_{j|k}$. Similarly, we denote the 
%$j^{\text{th}}$ state estimate and measurement disturbances as $\hat{w}_{j|k}$ 
%and $\hat{v}_{j|k}$, respectively, for $j \in \mathbb{I}_{0:k-1}$.
At time instant $k>N_k$, consider the cost function as follows
\begin{equation}\label{eq:mhe-algo}
\begin{aligned}
V_k(\chi_{k-N|k}, \boldsymbol{\omega}) := & \min_{\chi_{k-N|k}, \omega, \nu} 
\, \ell_{k-N|k}(\chi_{k-N|k} - \bar{x}_{k-N}) 
\\
&+ \sum_{i=k-N}^{k-1} \ell(\omega_{i|k}, \nu_{i|k}) \\
\text{subject to } 
& \chi_{i+1|k} = f(\chi_{i|k}, \mu_{i|k}) + \omega_{i|k} \\
& y_i = h(\chi_{i|k}, \mu_{i|k}) + \nu_{i|k}
\end{aligned}
\end{equation}
where the prior estimate $\bar{x}_{k-N}$ at each sampling time $k$ is propagated based on the state function and the last estimated state $\hat{x}_{k-N_k-1|k}$. $\mathbf{y}_{N_k} = \begin{bmatrix} y_{k-N|k}^T, \dots, y_{k|k}^T \end{bmatrix}^T$ 
is a sequence of measurements from moment $k - N$ to $k - 1$. The prior weighting matrix and the stage cost matrix are defined as: $\ell_{k-N|k}(\cdot)=100\cdot \mathbb I_4$ and $\ell(\cdot)=200\cdot \mathbb I_4$. Due to the computational cost of MHE, a relatively large sampling time was chosen for the simulation. The resulting state trajectories are shown in Figure~\ref{fig:mhe}, corresponding to a sampling time of $\Delta t=0.05$, and demonstrating that MHE provides a reasonable estimate of both angle and velocity states. 

\begin{remark}
  We did not compare the estimation plots with the other approaches for two
  reasons. Firstly, because of the repeated solution of the optimization problem
  in \eqref{eq:mhe-algo}, the algorithm is computationally demanding so that
  only a rougher time discretization could be chosen, which, however, means a
  worse performance in the estimation. Secondly, because of the different time
  step, the noise realization in the simulation led to a slightly different
  nominal trajectory so that a direct comparison is not possible. However, we
  have added the $\mathcal L_2$ estimation error in Table \ref{tab 2 rmse} to
  make the comparison. There, we also confirm that the performance hinges on the
  time step size and provide the much improved performance on a finer time grid
  too.
\end{remark}
%figure 
\begin{figure}[t]
  \centering
  \includegraphics[width=\linewidth]{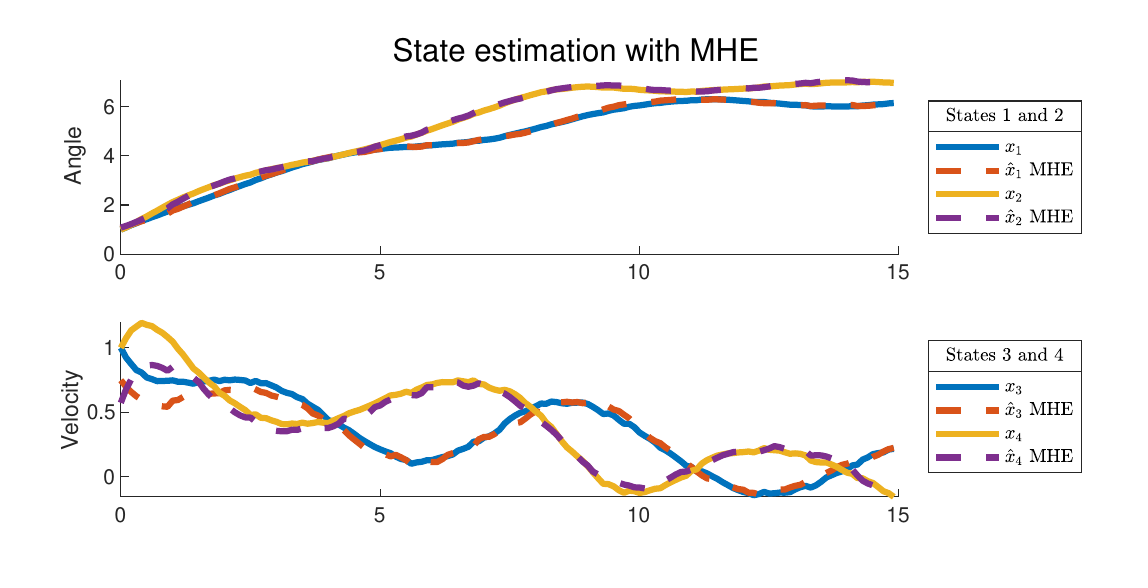} 
  \caption{True and estimated state trajectories under MHE}
  \label{fig:mhe}
\end{figure}
%--------------------------

\begin{figure}[t]
    \centering
    \includegraphics[width=\linewidth]{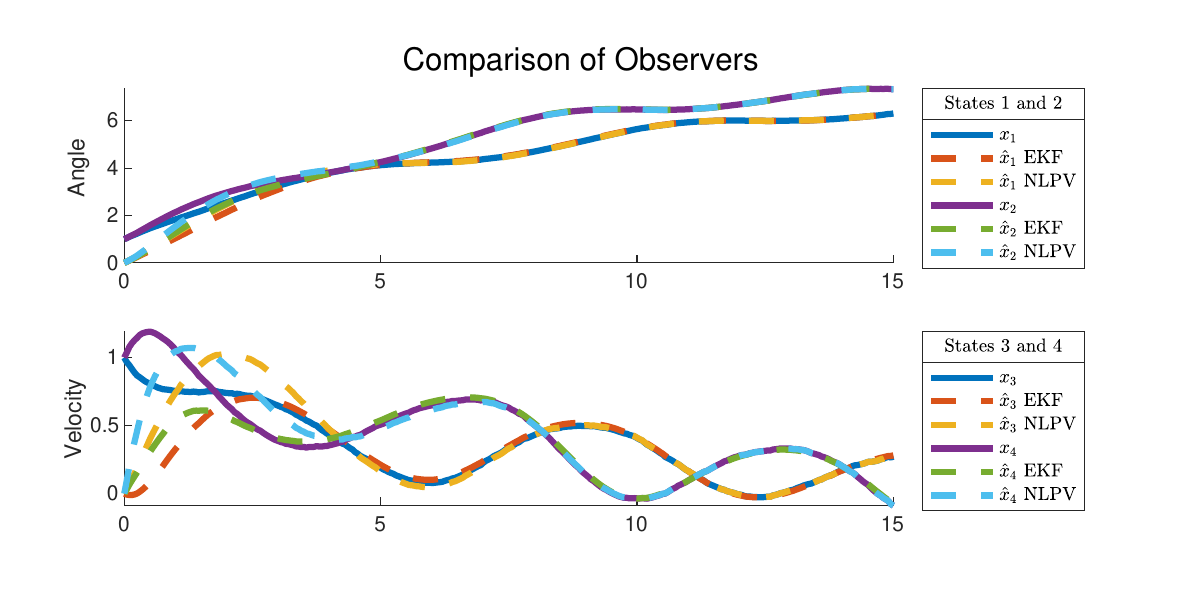}
    \caption{Illustration of estimated state trajectories v.s. actual state
    trajectories. }
    %TODO: need to export these pictures properly (best as vector graphics). Now they have different fonts and they are blurry if one zooms in.}
    \label{fig 1 state vs time NLPV obs}
\end{figure}
\subsection{NLPV observer}
The NLPV model of a two-link robotic arm~\eqref{eq:2arm-nonlin NLPV form} under the presence of noise/disturbance is illustrated as:
\begin{equation*}%\label{eq:2arm-nonlin NLPV form noise}
    \begin{split}
        \dot{x}(t)&= A(\theta(t)) x(t)+ G(\theta(t)) f(x(t))+B(\theta(t)) u(t)+E \omega,\\
        y(t)&=C x(t)+D \omega,
    \end{split}
\end{equation*}
where $\omega= \mathcal{N}(0,0.1)$ represents the noise/disturbance affecting the system dynamics and output. Further, $E=\begin{bmatrix}
    1&1&1&1
\end{bmatrix}^\top$ and $D=\begin{bmatrix}
   1&1
\end{bmatrix}^\top$. The remaining parameters and variables are the same as in~\eqref{eq:2arm-nonlin NLPV form}. 

The general conceptual and practical approach to treat quasi LPV systems with
LPV methods bases on the assumption that the range of the parameter $\theta(x)$
is bounded with known bounds and that, at runtime, the current parameter value
$\theta(t)$ is obtained by $\theta(x(t))$.

\begin{itemize}
  \item[1.] As for the bounds, we run example simulations to determine the
    bounds for $\theta_{\min}
    \leq \theta(x(t)) \leq \theta_{\max}$
    empirically. We note that, oftentimes, bounds can be derived analytically because of known
    bounds for $\cos$, $\sin$, $\sin(x)/x$.
  % \item[1.] Use the bounds to transform the affine LPV representation into
  %   polytopic form.
  \item[2.] As for the state dependencies, we note that in our formulation, the
    $\theta(x)$ only depends on measured quantities. % Simply use $\rho(t) = \rho(x(t))$ in the simulations.
\end{itemize}

The observation of nominal trajectory led to the following partially empirical
bounds on the parameters: 
$\theta_{1_{\min}}=-0.1$,~$\theta_{1_{\max}}=1$,~$\theta_{2_{\min}}=-1$,~$\theta_{2_{\max}}=0.6$,~$\theta_{3_{\min}}=-0.5$,~$\theta_{3_{\max}}=0.5$,~$h_{{\min}}=0.098$ and~$h_{{\max}}=0.125$.

The observer~\eqref{NLPV obs} is employed in the MATLAB environment.
The initial value of the system and observer are as follows: $x_0=\begin{bmatrix}
    1 &1 &1&1 
\end{bmatrix}^\top$ and $\hat{x}_0=\begin{bmatrix}
    0 &0 &0&0 
\end{bmatrix}^\top$
In order to compute the observer parameters, the authors have solved the proposed LMI~\eqref{NLPV e H inf LMI} using YALMIP solver(-{SDP}) in MATLAB 2023b. 
Further, the obtained LMI solution is compared with the method proposed in~\cite{LPV_2020_RNC} and summarized in Table~\ref{Tab 1 LMI comparison}. It clearly indicates that the proposed LMI is feasible for the aforementioned system whereas the LMI of~\cite[Theorem 1]{LPV_2020_RNC} provides an infeasible solution.
Further, the gain parameters obtained from the solution of LMI~\eqref {NLPV e H inf LMI} are illustrated as follows:
\begin{equation*}
\begin{aligned}
    L_0&=10^4\times \begin{bmatrix}
    0.0130  & -0.0129\\
   -0.0058  &  0.0059\\
    0.4528  & -0.4527\\
   -1.4347  &  1.4348
\end{bmatrix},
~L_1=10^4\times \begin{bmatrix}
    0.0116 &  -0.0116\\
   -0.0041  &  0.0041\\
    0.1301  & -0.1301\\
   -1.1561  &  1.1561
\end{bmatrix},~\\
L_2&=10^3\times \begin{bmatrix}
   -0.8358 &   0.8360\\
    0.4481 &  -0.4479\\
 -165.3136 & 165.3139\\
  628.5105 &-628.5112
\end{bmatrix},~
L_3=10^3\times \begin{bmatrix}
   -0.0282 &   0.0281\\
    0.1353 &  -0.1354\\
  -90.4052 &  90.4064\\
  304.1313 &-304.1381
\end{bmatrix}.
\end{aligned}
\end{equation*}
\begin{table}[t]
\centering
\caption{Comparison of LMI performance}
%\resizebox{\columnwidth}{!}
{%
\begin{tabular}{|c|c|c|c|}
\hline
\multirow{2}{*}{No.} & \multirow{2}{*}{\begin{tabular}[c]{@{}c@{}}Noise attenuation\\ parameters\end{tabular}} & \multicolumn{2}{|c|}{LMI methods} \\ \cline{3-4} 
 &  & LMI~\eqref{NLPV e H inf LMI} & \cite[LMI~(24)]{LPV_2020_RNC} \\ \hline
1 & $\sqrt{\mu}$ & \multicolumn{1}{c|}{$3.6085\times 10^{-05}$} & Infeasible \\ \hline
\end{tabular}%
}
\label{Tab 1 LMI comparison}
\end{table}
% \todo[inline]{Jan: Question: What if we put LPV observer result as infeasible solution in the paper and highlight that LPV observer is not feasible at this stage. That's the reason we need to go for NLPV approach or the enhancement of the existing LPV approach as a future scope. or go for new technique}

Further, the observer~\eqref{NLPV obs} is implemented in the MATLAB environment for the state estimation through the use of the aforementioned gain matrices. Figure~\ref{fig 1 state vs time NLPV obs} represents the plots of estimated states and actual states. It shows that the proposed observer attenuates noise/disturbance present in the dynamics and output, and provides an accurate state estimation. The behavior of estimation error is described in Figure~\ref{fig 2 error vs time NLPV obs}. It clearly shows the asymptotic convergence of estimation error and efficient noise attenuation. For more clarity, the authors have computed the RMSE values of estimation error for duration of $5\leq t \leq 15$, and illustrated in Table~\ref{tab 2 rmse}.
\begin{figure}[t]
    \centering
   \includegraphics[width=\linewidth]{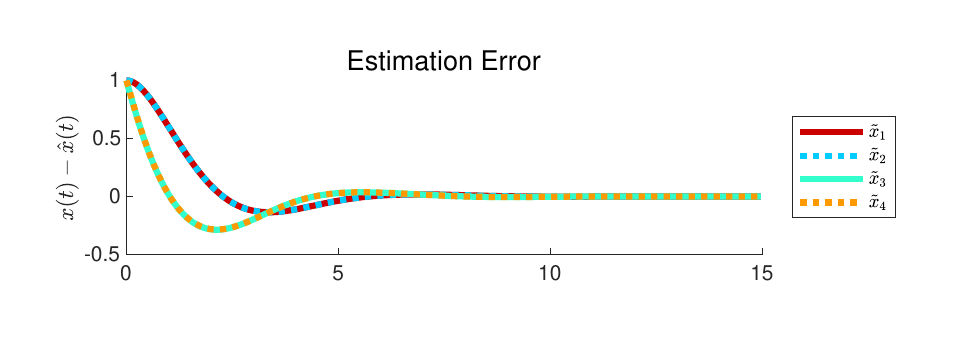} 
    \caption{Graph of estimation error in case of NLPV observer. }
    %TODO: this is too large -- use same style as the other plots.
    \label{fig 2 error vs time NLPV obs}
\end{figure}
\begin{table*}[t]
\begin{center}
\caption{$\mathcal L_2$ values of estimation error $\tilde{x_i}$, $i=1,2,3,4$
for the different nonlinear approaches on different observation horizons.}
\label{tab 2 rmse}
\renewcommand{\arraystretch}{1.1}
\begin{tabular}{|c|c|c|c|c|c|c|c|c|}
\hline
& \multicolumn{4}{|c|}{Observation interval: $0\leq t \leq 15$} & \multicolumn{4}{|c|}{Observation interval: $5\leq t \leq 15$}\\ \hline
& $\tilde{x}_1$&$\tilde{x}_2$&$\tilde{x}_3$&$\tilde{x}_4$ & $\tilde{x}_1$&$\tilde{x}_2$&$\tilde{x}_3$&$\tilde{x}_4$  \\ \hline
NLPV &   \begin{tabular}{c}
       $0.1885$
   \end{tabular}  & \begin{tabular}{c}
       $0.1885$
   \end{tabular}  &  \begin{tabular}{c}
       $0.1269$
   \end{tabular}&\begin{tabular}{c}
       $0.1273$
   \end{tabular}
   &$0.0278$ & $0.0278$ &$0.0435$ &$0.0435$ \\ \hline
% NLPV observer of~\cite{LPV_2020_RNC} & \multicolumn{8}{|c|}{Infeasible solution} \\ \hline
% ~\textcolor{red}{LPV Observer}~%\cite{} & \multicolumn{8}{|c|}{Infeasible solution} \\ \hline
EKF~%\cite{} 
& $0.2310$& $0.2468$& $0.1557$&$0.1804$ & $0.0307$&$0.0311 $&$0.0361$&$0.0378$   \\ \hline
MHE ($\Delta t=0.1$)~%\cite{} 
& $0.2101$&$0.2267$&$0.3338$&$0.4401$ & $0.1413$& $0.1327$&$0.0972$ & $0.1091$
\\  \hline
MHE ($\Delta t=0.05$)~%\cite{} 
& $0.2046$&$0.2426$&$0.2679$&$0.2427$ & $0.0812$&$0.0759$ &$0.0941$ &$0.1012$\\  \hline
\end{tabular}
%\vspace{-0.5cm}
\end{center}
\end{table*}
% \todo[inline]{JH: suggestion: You can just insert the values of RMSE wrt to method (either EKF or MHE). And please include there ref also.}

\section{Conclusion}

In this work, we have examplified and applied the NLPV factorization and
observer design to a two-link robot arm model. The NLPV approach showed its
advantages over plain LPV factorizations by a simpler scheduling map that only
depends on measured quantities. In the numerical simulation, the NLPV observer
competed well with the established EKF approach.

Future work will concern robustness issues that will arise if the variables
needed for the NLPV factorization need to be estimated too like in the case that
they are not available through the system's output.

%--------------------------------------------------
% \begin{thebibliography}{99}

% \bibitem{1} Spiegel, M. R. (1981). Theory and problems of Advanced Calculus: Si (metric) edition. McGraw-Hill. 

% \end{thebibliography}
\bibliographystyle{abbrv}
\bibliography{references}

\end{document}